\documentclass[12pt]{amsart}
\usepackage{amsfonts,amssymb}
\usepackage{amsmath}
\newtheorem{theo}{Theorem}[section]

\newtheorem{lemma}[theo]{Lemma}
\newtheorem{cor}[theo]{Corollary}
\newtheorem{fact}[theo]{Fact}

\newcommand{\spV}{M}

\makeatletter
\@namedef{subjclassname@2020}{%
  \textup{2020} Mathematics Subject Classification}
\makeatother

\begin{document}

\title{On irreducible representations of free group}
\author{Waldemar Hebisch}
\address{Waldemar Hebisch \\Mathematical Institute \\
University of Wroc{\l}aw \\
50-384 Wroc{\l}aw \\
pl. Grunwaldzki 2/4 \\
Poland}
\email{hebisch@mail.math.uni.wroc.pl}

% \maketitle

\begin{abstract}
 We prove
that for a suitable class of representations of free group tensor
products are generically irreducible.  In particular
we prove that there exist irreducible boundary
realizations with infinite dimensional fiber.
\end{abstract}


\maketitle

\section{Introduction}

General structure of space of irreducible unitary representations
of group (or $C^{*}$-algebra)
of type I is now well understood.  For groups of type II
it is known that structure of irreducible representations is
less regular and our understanding of representations of
such groups is quite limited.  Free group is a classical
example of group of type II and there is considerable
interest in various families of representations of free
group.  Of particular interest are representations
weakly contained in the regular representation.  We prove
that for a suitable class of representations tensor
products are generically irreducible, which implies
that corresponding representations must be less
regular than previously believed.  In particular
we prove that there exist irreducible boundary
realizations with infinite dimensional fiber.

\section{Preliminaries and notation}

Let $S$ be a finite set with cardinality bigger than $1$ and
$G$ be the free group generated by $S$.  We denote by
${\mathcal A} = {\mathbb C}[G]$ group algebra of $G$.
We say that unitary representation $\eta$ of $G$ is {\it tempered} when
it is is weakly contained in the regular representation $\lambda$ of $G$.

It is well-known that an irreducible tempered unitary representation $\eta$ of
$G$ is weakly equivalent to the regular representation, that is for all
$f \in {\mathcal A}$ we have
$$\|\eta(f)\| = \|\lambda(f)\|.$$

Let $\partial G$ be boundary of $G$.  It is well known that for
tempered unitary representation $\eta$ we can view representation space
$V$ as $L^2(\mu, H)$ where  $\mu$ a probability measure on $\partial G$
which is quasi invariant under action of $G$, $H$ is a Hibert space
and action of $g\in G$ consists of translation by $g$ and multiplication
by an operator depending on $\omega \in \partial G$ and $g\in G$
(see \cite{HKS} Proposition 2.2).
Such view is called {\it boundary realization}.  We call space
$H$ {\it fiber} of the boundary realization.  There is
powerful machinery to prove irreducibility of nice tempered
representations with boundary realizations having finite dimensional
fiber (see \cite{HKS} for more technical part, \cite{PSt},
\cite{KS04}, %\cite{KSS16},
\cite{KSS23} and other for applications).
One of goals of
current paper is to prove existence of
irreducible representations of $G$ having boundary realizations
with infinite dimensional fiber.  Namely, given a boundary
realization we can form tensor product with arbitrary unitary
representation $\pi$ on Hilbert space $M$.  Then new boundary
realization has fiber $H\otimes M$.  So for infinite dimensional
$M$ we get infinite dimensional fiber.  Our main result below
is that for some $\eta$ tensor product $\eta \otimes \pi$
is irreducible for generic $\pi$.

In proof of our results we need Kaplansky density theorem:

\begin{fact}
Let $H$ be a Hilbert space.
If $A$ is a selfadjoint subalgebra of $B(H)$, then unit ball of
$B$ is strongly dense in unit ball of $\bar A$, where $\bar A$
means strong closure of $A$ in $B(H)$.
\end{fact}

\section{Main result}

Let ${\mathcal A} = {\mathbb C}[G]$ be group algebra of free
group.  Assume that $\eta$ is a tempered unitary representation of $G$,
such that for every irreducible finite dimensional unitary representation
$\pi$ of $G$ tensor product $\eta\otimes\pi$ is irreducible.

Let $\spV$ be a fixed infinite dimensional separable Hilbert
space.
Representation of $G$ on $\spV$ is uniquely determined by
images of generators of $G$, and since $G$ is a free group,
any assignment of unitary operators to generators extends
to a representation.  So we can identify set of representations
of $G$ on $\spV$ with tuples of unitary operators on $\spV$.
On unitary operators we have topology of strong operator
convergence.  This topology is separable (there is a countable
dense subset) and can be given by a complete
metric.  Identifying representations with tuples of
operators we can treat space of representations as a
complete separable metric space.

Let $H$ be representation space of $\eta$.  Consider
$H\otimes \spV$ where tensor product is in category of Hilbert
spaces.

\begin{theo}\label{gen:irr}
Set of representations $\pi$ such that $\eta\otimes\pi$
is irreducible is of second category.
\end{theo}

For the proof we need some lemmas and definitions.

  Fix a countable dense sequence $x_n$, $n \in \mathbb N$ of elements
of $H\otimes \spV$, such that for each $n$ we have $x_n \ne 0$.

Put
$$
U_{j, k, \delta, f} = \{\pi : \|(\eta\otimes\pi)(f)x_j - x_k\|
 <  \delta \},
$$
$$
V_{j, k, \delta} = \bigcup_{f: \|\eta(f)\| < 4\|x_k\|/\|x_j\|}
U_{j, k, \delta, f},
$$
$$
W = \bigcap_{k, j\in {\mathbb N}, \delta \in {\mathbb Q}_+} V_{j, k, \delta}.
$$

\begin{lemma}\label{fin-appr}
For fixed $\delta > 0$, $j, k$ there exists finite dimensional
subspace $\spV_m \subset \spV$ such that for all 
finite dimensional subspaces $\spV_n$ such that
$\spV_m \subset \spV_n \subset \spV$ set
$V_{j, k, \delta}$
contains all representations of form $\pi = \pi_n\oplus Id$
where $\pi_n$ is a finite dimensional irreducible representation
on $\spV_n$, and $Id$ maps all elements of $G$ to identity on
orthogonal complement of $\spV_n$.
\end{lemma}

\begin{proof}
% Clearly $U_{j, k, \delta, f}$ is open, so also $V_{j, k, \delta}$ is open.
% So it remains to show that $V_{j, k, \delta}$ is dense.  We
% will show that if $\spV_n$ is large enough finite dimensional subspace
% of $\spV$ then $V_{j, k, \delta}$ contains
% set of representations $\pi$ of form $\pi = \pi_n\oplus Id$
% where $\pi_n$ is a finite dimensional irreducible representation
% on $\spV_n$, and $Id$ maps all elements of $G$ to identity on
% orthogonal complement of $\spV_n$.  Since set of such representations
% is dense, this implies that $V_{j, k, \delta}$ is dense.
Let $x_j'$ (respectively $x_k'$) be orthogonal projection
of $x_j$ (respectively $x_k$) onto $H\otimes \spV_n$.
Choose $\spV_m$ such that $\|x_k - x_k'\| < \delta/3$,
$\|x_j - x_j'\| < \frac{\delta\|x_j\|}{12\|x_k\|}$
and $\|x_k'\|/\|x_j'\| < 2\|x_k\|/\|x_j\|$ (once it holds
for projection onto $H\otimes \spV_m$ it also holds for projection
onto larger space $H\otimes \spV_n$).

By assumption $\eta\otimes\pi_n$ is irreducible, so
$(\eta\otimes\pi_n)({\mathcal A})$ is strongly dense in
$B(H\otimes \spV_n)$.
By Kaplansky density theorem there exist
$f \in {\mathcal A}$ such that
$$
\|(\eta\otimes\pi_n)(f)x_j' - x_k'\| < \frac{\delta}{3}
$$
and
$$
\|\eta\otimes\pi_n)(f)\| \leq 2\frac{\|x_k'\|}{\|x_j'\|}
< 4\frac{\|x_k\|}{\|x_j\|}.
$$
It is well known that tempered irreducible representations
of free group
% weakly contained in the regular representation
are weakly equivalent to regular
representation.  So $\eta\otimes\pi_n$ is weakly
equivalent to regular representation which is
weakly equivalent to $\eta$.  Consequently
$$
\|\eta(f)\| = \|(\eta\otimes\pi_n)(f)\| \leq 4\frac{\|x_k\|}{\|x_j\|}.
$$
and
$$
\|(\eta\otimes Id)(f)\| \leq 4\frac{\|x_k\|}{\|x_j\|}.
$$
Now
$$
\|(\eta\otimes(\pi_n\oplus Id))(f)x_j - x_k\|
$$
$$
\leq \|(\eta\otimes\pi_n)(f)x_j' - x_k'\| +
\|(\eta\otimes Id)(f)\|\|x_j - x_j'\| + \|x_k - x_k'\|.
$$
By choice of $\spV_n$ third term is less than $\delta/3$.
By choice of $f$ first term is less than $\delta/3$.
For the middle term we have
$$
\|(\eta\otimes Id)(f)\|\|x_j - x_j'\| \leq
\frac{4\|x_k\|}{\|x_j\|}\|x_j - x_j'\|
$$
which by choice of $\spV_n$ is less than $\delta/3$.
So $\pi_n\oplus Id \in U_{j, k, \delta, f}$.   Our estimate
on $\|\eta(f)\|$ means that $U_{j, k, \delta, f}$ will appear
is sum defining $V_{j, k, \delta}$, so
$\pi_n\oplus Id \in V_{j, k, \delta}$.
\end{proof}

\begin{lemma}
If $\delta > 0$, then
$V_{j, k, \delta}$ is open dense subset in the space of representations.
\end{lemma}

\begin{proof}
Clearly $U_{j, k, \delta, f}$ is open, so also $V_{j, k, \delta}$ is open.
So it remains to show that $V_{j, k, \delta}$ is dense.
But this follows from Lemma \ref{fin-appr}: representations
of form $\pi_n\oplus Id$ form a dense subset in space of all
representations.
\end{proof}

\begin{lemma}
If $\pi \in W$, then $\eta\otimes\pi$ is irreducible.
\end{lemma}

\begin{proof}
We need to prove that every nonzero $v \in H\otimes \spV$ is
cyclic, that is ${\mathcal A}v$ is dense.  Let $y \in H\otimes \spV$.
Fix $\epsilon > 0$.  Let
$$
\delta_1 = \min(\frac{\epsilon \|v\|}{48\|y\|}, \|v\|/2).
$$
Since sequence $x_n$ is dense we can find $j$ such that
$\|v - x_j\| < \delta_1$.  Note that due to choice of
$\delta_1$ we have $\|v\| \leq 2\|x_j\|$.
Let $\delta_2 = \epsilon/3$.  Since
$x_n$ is dense we can find $k$ such that 
$\|y - x_k\| < \min(\delta_2, \|y|/2)$.
Note that $2\|y\| \geq \|x_k\|$.
Since rationals are dense in real numbers there is
$\delta \in {\mathbb Q}_+$ such that $\delta \leq \delta_2$.
By definition $W \subset V_{j, k, \delta}$, so there is
$f \in {\mathcal A}$ such that $\pi \in U_{j, k, f, \delta}$
and $\|\eta(f)\| \leq 4\|x_k\|/\|x_j\|$.
But since $\eta\otimes\pi$ is weakly contained in $\eta$
condition on $f$ means
$$
\|(\eta\otimes\pi)(f)\|\delta_1
\leq \|\eta(f)\|\frac{\epsilon \|v\|}{48\|y\|}
\leq \frac{4\|x_k\|}{\|x_j\|}\frac{2\epsilon \|x_j\|}{48\|x_k\|/2}
 = \frac{\epsilon}{3}.
$$
$\pi \in U_{j, k, f, \delta}$ means
$$
\|(\eta\otimes\pi)(f)x_k - x_j\| < \delta \leq \frac{\epsilon}{3}. 
$$
Now
$$
\|(\eta\otimes\pi)(f)v - y\| \leq
\|(\eta\otimes\pi)(f)(v - x_j)\| + \|(\eta\otimes\pi)(f)x_j - x_k\|
+ \|y - x_k\|
$$
$$
< \|(\eta\otimes\pi)(f)\|\|v - x_j\| + \frac{\epsilon}{3}
+ \frac{\epsilon}{3} \leq \|(\eta\otimes\pi)(f)\|\delta_1
+ \frac{2\epsilon}{3} \leq \frac{\epsilon}{3} + \frac{2\epsilon}{3}
= \epsilon.
$$
\end{proof}

Generalization: Above we can replace space of unitary operators
with strong operator topology by a subset $U_s$ with stronger topology,
as long as set of operators of form $I + A$ where $A$ is finite
dimensional operator is dense in $U_s$.  For example, we can
take as $U_s$ set of operators of form $I + A$ where $A$ is
compact (or of trace class).  Intersection of $V_{j, k, \delta}$
with $U_s$ is clearly open subset of $U_s$ since topology
on $U_s$ is stronger.  Our density arguments works via
approximation by representations of form $\pi_n\oplus Id$
and works as long as $U_s$ has appropriate approximation
property.

\begin{cor}\label{inf:fib} There exists irreducible boundary realization with
infinite dimensional fiber.
\end{cor}
\begin{proof}
In \cite{PSt} it is proven that representations $\eta$ in
anisotropic principal series have property in assumption of
theorem \ref{gen:irr}, that is tensor product with any
finite dimensional representation remains irreducible.
So, by \ref{gen:irr} there exist infinite dimensional representation
$\pi$ such that $\eta\otimes \pi$ is irreducible.  Of course
tensor product of a boundary realization of $\eta$ with $\pi$ gives
boundary realization of $\eta\otimes \pi$ with infinite
dimensional fiber.
\end{proof}

In \cite{HKS} crucial technical assumption is finite trace
condition (FTC) which requires that certain operators
build from representation have finite trace (for details see
\cite{HKS} Definition 5.9 and \cite{PSt} Definition 2.7).

\begin{cor}
There exists irreducible representation having two perfect boundary
realizations which do not satisfy finite trace condition.
\end{cor}
\begin{proof}
Like in proof of Corollary \ref{inf:fib}
we take $\eta$ form anisotropic principal series.  By \cite{PSt}
$\eta$ has two perfect boundary realizations which satisfy finite
trace condition.  Take infinite dimensional $\pi$ such
that $\eta\otimes \pi$ is irreducible.  Inspection of proof
of \cite{PSt} Lemma 2.8 shows that corresponding pair of
boundary realizations of $\eta\otimes \pi$ does not satisfy
finite trace condition.
\end{proof}

{Remark}.  One can prove the last corollary on a quite different
way, by considering restriction of representations of $PSL(2, {\mathbb R})$
to an embedded free group.  However, we give different way in
which FTC can fail and our genericity result indicates that
that this failure mode is probably quite common.

\section{Final remarks}

We could weaken assumption of $\eta$ in Theorem \ref{gen:irr}
to the following: for each $n$ set of $n$-dimensional irreducible unitary
representations of $G$ such that $\eta\otimes \pi$ is irreducible
is of second category.

Theorem \ref{gen:irr} can be formulated in the following way:
if $\eta$ is a tempered irreducible unitary representation of $G$ satisfying
certain extra condition, then set of unitary representations $\pi$
such that hat $\eta\otimes \pi$ is irreducible
is of second category.  It is natural to ask if extra conditions
are really necessary, maybe the result remains true for all
tempered irreducible unitary representations of $G$?  We do
not know the answer.  However, by our arguments the question
has positive answer if for each $n$ answer for $\pi$ of dimension $n$
is positive.

{\bf Acknowledgements:} I would like to thank Tim Steger for
fruitful discussions.

\end{document}